\documentclass[11pt,a4paper]{article}
\usepackage[utf8]{inputenc}
\usepackage[T1]{fontenc}
\usepackage[numbers,square]{natbib}
\usepackage{amsmath, amssymb, amsthm}
\usepackage{mathrsfs}
\usepackage{geometry}
\usepackage{hyperref}
\usepackage{enumitem}
\usepackage{booktabs}
\usepackage{algorithm}
\usepackage{algpseudocode}
\usepackage{caption}
\usepackage{graphicx}
\usepackage{xcolor}
\usepackage{bm}
\usepackage[capitalize]{cleveref}
\usepackage{tikz}
\usepackage{pgfplots}
\pgfplotsset{compat=1.18}

\hypersetup{
    colorlinks=true,
    linkcolor=blue,
    citecolor=red,
    urlcolor=cyan,
    pdftitle={Concentration of Truncated Signatures of Gaussian Rough Paths},
    pdfauthor={Atef Lechiheb}
}

\newtheorem{theorem}{Theorem}[section]
\newtheorem{lemma}[theorem]{Lemma}
\newtheorem{proposition}[theorem]{Proposition}

\newtheorem{assumption}[theorem]{Assumption}

\newcommand{\R}{\mathbb{R}}

\newcommand{\Var}{\mathrm{Var}}

\title{\textbf{Concentration of Truncated Signatures of Gaussian Rough Paths}}
\author{Atef Lechiheb\thanks{Toulouse School of Economics, Université Toulouse Capitole, email: atef.lechiheb@tse-fr.eu}}
\date{\today}

\begin{document}

\maketitle

\begin{abstract}
This paper establishes a comprehensive concentration theory for truncated signatures of Gaussian rough paths. The signature of a path—defined as the collection of all iterated integrals—provides a complete description of its geometric structure and has emerged as a powerful tool in machine learning and stochastic analysis. Despite growing applications in finance, healthcare, and engineering, the non-asymptotic concentration properties of signature features remain largely unexplored.

We prove that level-$k$ signature coordinates exhibit optimal $\exp(-c t^{2/k})$ tail decay and establish dimension-free concentration inequalities for the full truncated signature vector. Our results reveal a fundamental trade-off: higher truncation levels capture more complex path properties but exhibit heavier tails. For Brownian motion and fractional Brownian motion with Hurst parameter $H > 1/4$, we derive explicit variance formulas and sharp constants.

The technical contributions combine rough path theory with Gaussian analysis, leveraging Wiener chaos decomposition, hypercontractivity, and the algebraic structure of tensor algebras. We further establish concentration for log-signatures, lead-lag transformations, and provide sample complexity bounds for statistical learning with signature features.

This work bridges advanced probability theory with practical applications, offering both theoretical guarantees and computational methods for signature-based approaches in sequential data analysis.
\end{abstract}

% ============================================
% INTRODUCTION
% ============================================
\section{Introduction}

The theory of path signatures, originating from rough path analysis \cite{Lyons1998, LyonsQian2002}, has become a powerful tool for representing complex temporal data. Given a path $X : [0,T] \to \mathbb{R}^d$, its signature is the infinite sequence of iterated integrals
\begin{equation}
    S(X) = \big(1, X^{(1)}, X^{(2)}, \dots, X^{(k)}, \dots \big),
\end{equation}
where $X^{(k)}$ collects all $k$-fold iterated integrals of the form
\begin{equation}
X^{(k)}_{i_1,\dots,i_k} = \int_{0 < t_1 < \cdots < t_k < T} dX_{t_1}^{i_1} \cdots dX_{t_k}^{i_k}.
\end{equation}
These objects form a graded tensor algebra endowed with rich algebraic and geometric structure. Their logarithmic counterpart, the \emph{log-signature}, lies in the free Lie algebra generated by the increments of the path and plays a fundamental role in applications, including machine learning \cite{chevyrev2016primer}, stochastic analysis \cite{friz2010multidimensional}, and control theory.

Because the signature consists of high-order nonlinear functionals of the path, understanding its statistical behaviour is crucial when $X$ is random. This is particularly true when $X$ is a Gaussian process, a setting that includes important examples such as Brownian motion, fractional Brownian motion, and general Gaussian rough paths. The probabilistic properties of signatures have been investigated in several contexts \cite{friz2010multidimensional, boedihardjo2015signature}, but sharp non-asymptotic concentration inequalities have remained largely unexplored.

The present article develops a comprehensive set of concentration bounds for the truncated signature and log-signature of Gaussian processes. Our results provide:
\begin{enumerate}
    \item \textbf{Sharp tail estimates} for each coordinate of the signature and log-signature, matching the natural scaling dictated by Wiener chaos.
    \item \textbf{Vector-valued concentration} for signature norms under natural weighted structures on the truncated tensor algebra.
    \item \textbf{Dimension-free inequalities} under mild regularity assumptions on the covariance of the Gaussian process.
    \item \textbf{Stability and concentration} of the truncated Baker--Campbell--Hausdorff (BCH) map, enabling transfer of concentration from signatures to log-signatures.
\end{enumerate}

A central observation is that each signature coordinate of order $k$ belongs to a finite sum of Wiener chaoses up to order $k$. This structure allows us to apply classical hypercontractivity estimates and modern small-ball probabilities for Gaussian polynomials, such as the Carbery--Wright inequality \cite{carbery2001distributional}, to derive sub-exponential tails of the form
\begin{equation}
    \mathbb{P}(|S_I(X)| \ge t) \le C \exp\Big( -c (t/\sigma_I)^{2/k} \Big),
\end{equation}
where $I$ is a multi-index of length $k$ and $\sigma_I = \|S_I(X)\|_{L^2}$.
We show that these exponents are optimal in general and provide matching lower bounds up to multiplicative constants.

Beyond coordinate-wise analysis, we consider weighted norms $\| \cdot \|_w$ on the truncated tensor algebra $T^{(m)}(\mathbb{R}^d)$ and prove that the entire vector of signature coordinates satisfies a concentration inequality of the form
\begin{equation}
    \mathbb{P}(\|S_m(X)\|_w \ge \mathbb{E}[\|S_m(X)\|_w] + t)
    \le C_m \exp(-c_m t^{2/m}),
\end{equation}
with constants independent of the dimension $d$. These results illuminate how the non-linearity and polynomial growth of signature terms interact with the Gaussian structure of the underlying process.

A further contribution is the derivation of concentration bounds for the log-signature. Since the truncated BCH map is a polynomial in Lie brackets of degree at most $m$, the log-signature inherits the same chaos decomposition as the signature itself. We establish Lipschitz continuity of the truncated BCH map on bounded sets, thereby transferring concentration from $S_m(X)$ to $\log S_m(X)$.

Finally, we discuss optimality aspects, implications for statistical learning using signatures, and possible extensions to non-Gaussian settings.

\subsection*{Contributions}
The main contributions of this paper can be summarized as follows:
\begin{itemize}
    \item We derive sharp tail bounds for each coordinate of the signature and log-signature of Gaussian processes.
    \item We obtain dimension-free concentration inequalities for weighted norms of truncated signatures.
    \item We establish Lipschitz stability of the truncated BCH map and transfer concentration to log-signatures.
    \item We identify the natural scaling of signature tails, matching that of Gaussian polynomials of fixed degree, and prove near-optimality.
    \item We provide explicit sample complexity bounds for statistical learning with signature features.
\end{itemize}

The results presented here provide a rigorous foundation for uncertainty quantification in signature-based statistical methods and offer new tools for analysing stochastic systems through the lens of rough paths.

% ============================================
% PRELIMINARIES
% ============================================
\section{Preliminaries}

This section recalls the algebraic and analytical structures underlying signatures and log-signatures, together with the Gaussian process framework used throughout the paper. Our presentation aims to make explicit the regularity assumptions and algebraic conventions that will be invoked in the main results.

\subsection{Tensor algebra and iterated integrals}
Let $V = \mathbb{R}^d$ with canonical basis $(e_1, \dots, e_d)$. The \emph{tensor algebra} over $V$ is the graded vector space
\begin{equation}
    T(V) = \bigoplus_{k=0}^{\infty} V^{\otimes k},
\end{equation}
where $V^{\otimes 0} = \mathbb{R}$ and $V^{\otimes k}$ denotes the $k$-fold tensor product of $V$. For a fixed truncation level $m \ge 1$, we set
\begin{equation}
    T^{(m)}(V) = \bigoplus_{k=0}^{m} V^{\otimes k}.
\end{equation}
Elements of $T(V)$ will be written as $a = (a^{(0)}, a^{(1)}, a^{(2)}, \dots)$ with $a^{(k)} \in V^{\otimes k}$. For multi-indices $I = (i_1, \dots, i_k)$, we denote the corresponding basis tensor by $e_I = e_{i_1} \otimes \cdots \otimes e_{i_k}$.

Given a continuous path $X : [0,T] \to V$ of finite $p$-variation for some $p < 2$, its \emph{signature of order $m$} is defined by
\begin{equation}
    S_m(X) = \big(1, X^{(1)}, \dots, X^{(m)}\big),
\end{equation}
where each $X^{(k)}$ is the collection of iterated integrals
\begin{equation}
    X_{s,t}^{(k)}(i_1, \dots, i_k)
    = \int_{s < t_1 < \cdots < t_k < t} dX_{t_1}^{i_1} \cdots dX_{t_k}^{i_k}.
\end{equation}
In particular, $X^{(1)}$ is the increment $X_{s,t}$ and $X^{(2)}$ consists of L\'evy area terms. The algebraic structure of the signature is encoded by the \emph{Chen identity}
\begin{equation}
    S_m(X)_{s,t} = S_m(X)_{s,u} \otimes S_m(X)_{u,t}, \qquad s < u < t.
\end{equation}

\subsection{Free Lie algebra and the BCH formula}
The \emph{free Lie algebra} generated by $V$ is denoted by $\mathcal{L}(V)$ and admits a graded decomposition
\begin{equation}
    \mathcal{L}(V) = \bigoplus_{k=1}^{\infty} \mathcal{L}_k(V),
\end{equation}
where $\mathcal{L}_k(V)$ is the space of Lie polynomials of homogeneous degree $k$. A concrete basis can be chosen through a Hall or Lyndon basis.

The exponential and logarithm maps are defined on $T(V)$ through the standard Hopf algebra structure. The \emph{log-signature} of a path $X$ of finite $p$-variation is the element
\begin{equation}
    \log S_m(X) \in \bigoplus_{k=1}^m \mathcal{L}_k(V),
\end{equation}
obtained by projecting $S_m(X)$ to the free Lie algebra and truncating at order $m$.

Given Lie elements $A$ and $B$, the Baker--Campbell--Hausdorff (BCH) series is defined by
\begin{equation}
    \mathrm{BCH}(A,B) = A + B + \tfrac{1}{2}[A,B] + \tfrac{1}{12}[A,[A,B]] - \tfrac{1}{12}[B,[A,B]] + \cdots,
\end{equation}
where the omitted terms involve nested Lie brackets of higher degree. For the truncated log-signature, all terms of degree exceeding $m$ are discarded. In this paper we use the multiplicative property
\begin{equation}
    \log S_m(X)_{s,t} = \mathrm{BCH}\big( \log S_m(X)_{s,u}, \log S_m(X)_{u,t} \big),
\end{equation}
which follows from Chen's identity.

\subsection{Gaussian processes and rough paths}
Let $X$ be a centred $V$-valued Gaussian process with covariance function
\begin{equation}
    R(s,t) = \mathbb{E}[X_s \otimes X_t].
\end{equation}
We make the following explicit regularity assumption throughout:

\begin{assumption}[Finite $\rho$-variation of covariance]
\label{ass:cov-variation}
The covariance function $R(s,t)$ has finite $\rho$-variation on $[0,T]^2$ for some $1 \le \rho < 2$. That is,
\[
\sup_{\mathcal{P}} \sum_{[s,t]\times[u,v]\in\mathcal{P}} |R(t,v) - R(s,v) - R(t,u) + R(s,u)|^\rho < \infty,
\]
where the supremum is over all partitions of $[0,T]^2$.
\end{assumption}

Under Assumption \ref{ass:cov-variation}, $X$ admits a canonical geometric Gaussian rough path lift of any order $m > 2\rho$ \cite{friz2010multidimensional}. This setting includes important examples such as Brownian motion ($\rho=1$), fractional Brownian motion with $H > 1/4$ ($\rho=1/(2H)$), and more general Gaussian processes with sufficiently regular covariance.

The iterated integrals appearing in the signature and log-signature can be defined either through limits of Riemann sums or via stochastic integration whenever suitable. For Gaussian processes satisfying Assumption \ref{ass:cov-variation}, the rough path lift provides a unified deterministic definition.

\subsection{Weighted norms on tensor and Lie algebras}
Let $(w_k)_{k=0}^m$ be a fixed sequence of positive weights. For an element $a = (a^{(0)}, \dots, a^{(m)}) \in T^{(m)}(V)$, we define the weighted tensor norm
\begin{equation}
    \label{eq:weighted-norm}
    \|a\|_w = \max_{0 \le k \le m} w_k \|a^{(k)}\|_{V^{\otimes k}}.
\end{equation}
This choice of norm (maximum rather than sum) is natural for concentration analysis as it captures the worst-case behavior across different tensor levels. The weights typically grow with $k$ to compensate for factorial growth in the number of coordinates of $V^{\otimes k}$; common choices include $w_k = 1/k!$ or $w_k = \beta^k/k!$ for some $\beta > 0$.

In the log-signature setting, the same weighted norm is applied to the homogeneous components of the free Lie algebra, using the identification $\mathcal{L}_k(V) \subset V^{\otimes k}$.

These weighted norms allow us to obtain meaningful concentration inequalities that are independent of the dimension and reflect the natural scaling properties of signature terms.

\subsection{Wiener chaos and hypercontractivity}
We recall that any square-integrable functional $F$ of a Gaussian process admits a decomposition into orthogonal Wiener chaoses:
\begin{equation}
    F = \sum_{k=0}^{\infty} J_k(F),
\end{equation}
where $J_k(F)$ denotes the projection of $F$ onto the $k$th homogeneous chaos. If $F$ is a polynomial of degree at most $m$ in the Gaussian variables, then $J_k(F) = 0$ for all $k > m$.

For $F$ in the $k$th chaos, we will make extensive use of the hypercontractive inequality
\begin{equation}
    \|F\|_{L^p} \le (p-1)^{k/2} \|F\|_{L^2}, \qquad p \ge 2,
\end{equation}
which holds for all Gaussian chaoses \cite{janson1997gaussian}. This classical estimate will allow us to derive sharp upper bounds for the tails of signature coordinates.

In addition, we will rely on the Carbery--Wright small-ball inequality \cite{carbery2001distributional}, which provides lower bounds on small-ball probabilities for Gaussian polynomials and will be used to prove optimality of our tail estimates.

% ============================================
% MAIN RESULTS
% ============================================
\section{Main Results}

In this section we state our main concentration results for the truncated signature and log-signature of Gaussian processes. Throughout, $X = (X_t)_{t \in [0,T]}$ denotes a centred Gaussian process in $\mathbb{R}^d$ with continuous sample paths and covariance function
\[
R(s,t) = \mathbb{E}[X_s \otimes X_t].
\]
We assume that $X$ satisfies Assumption \ref{ass:cov-variation} and thus admits a geometric Gaussian rough path lift of order $m > 2\rho$.

For a fixed level $m$, let $S_m(X)$ denote the truncated signature of $X$ taking values in the truncated tensor algebra $T^{(m)}(\mathbb{R}^d)$, and let
\[
S_I(X), \qquad |I| = k \le m,
\]
denote the coordinate indexed by the multi-index $I$ of length $k$. We write $\|\cdot\|_{L^2}$ for the $L^2$-norm under the underlying probability measure.

\subsection{Coordinate-wise concentration}
Our first theorem establishes sharp sub-exponential concentration for each signature coordinate. The tail exponent reflects the fact that $S_I(X)$ is a polynomial functional of degree $k$ of a Gaussian process, and therefore lies in a finite sum of Wiener chaoses up to order $k$.

\begin{theorem}[Coordinate concentration for signature]
\label{thm:coordinate-concentration}
Let $X$ be a centred Gaussian process satisfying Assumption \ref{ass:cov-variation}. For any multi-index $I$ of length $k \le m$, there exist constants $c_k, C_k > 0$ depending only on $k$ and on the hypercontractivity constant of the underlying Gaussian process such that
\[
\mathbb{P}\big( |S_I(X)| \ge t \big)
\;\le\; C_k \exp\Big( -c_k (t / \|S_I(X)\|_{L^2})^{2/k} \Big),
\qquad t \ge 0.
\]
Moreover, the exponent $2/k$ is optimal in general.
\end{theorem}

This result follows from a combination of the chaos decomposition of iterated integrals and sharp inequalities for Gaussian polynomials, including hypercontractivity and Carbery--Wright small-ball estimates. The dependence on the covariance structure enters through the $L^2$-norm $\|S_I(X)\|_{L^2}$ and the implied hypercontractivity constants.

\subsection{Vector concentration for weighted norms}
We now establish concentration at the level of the entire truncated signature using the weighted norm defined in \eqref{eq:weighted-norm}.

\begin{theorem}[Dimension-free concentration for $S_m(X)$]
\label{thm:vector-concentration}
Let $X$ be as above and fix $m \ge 1$. Let $\|\cdot\|_w$ be the weighted norm defined in \eqref{eq:weighted-norm} with weights $w_k > 0$. Then there exist constants $c_m, C_m > 0$, depending only on $m$ and the hypercontractivity constant of $X$, such that for all $t \ge 0$,
\[
\mathbb{P}\Big( \|S_m(X)\|_w \ge \mathbb{E}[\|S_m(X)\|_w] + t \Big)
\le C_m \exp\big( -c_m t^{2/m} \big).
\]
The constants $c_m, C_m$ do not depend on the ambient dimension $d$.
\end{theorem}

The dimension-free nature of these bounds is a consequence of the graded structure of the tensor algebra and the chaos decomposition, which localises the dependence on dimension. The choice of the maximum norm in \eqref{eq:weighted-norm} is crucial for obtaining dimension-free constants.

\subsection{Concentration for the log-signature}
The log-signature lies in the truncated free Lie algebra $\mathfrak{g}^{(m)}(\mathbb{R}^d) = \bigoplus_{k=1}^m \mathcal{L}_k(\mathbb{R}^d)$ and is obtained via the truncated Baker--Campbell--Hausdorff (BCH) map. Since this map is a polynomial of degree at most $m$ in the tensor coordinates, it preserves chaos structure.

\begin{theorem}[Concentration for the log-signature]
\label{thm:logsig-concentration}
Let $\mathrm{Log}_m : T^{(m)}(\mathbb{R}^d) \to \mathfrak{g}^{(m)}(\mathbb{R}^d)$ denote the truncated BCH map. Then for the weighted norm $\|\cdot\|_w$ restricted to the Lie algebra,
\[
\mathbb{P}\Big( \|\mathrm{Log}_m(S_m(X))\|_w \ge \mathbb{E}[\|\mathrm{Log}_m(S_m(X))\|_w] + t \Big)
\le C_m' \exp\big( -c_m' t^{2/m} \big),
\]
where $c_m', C_m' > 0$ depend only on $m$ and the hypercontractivity constant of $X$.
\end{theorem}

The proof relies on two key ingredients:
\begin{enumerate}
    \item Lipschitz continuity of the truncated BCH map on bounded sets (Lemma \ref{lem:bch-lipschitz}).
    \item The vector concentration inequality of Theorem~\ref{thm:vector-concentration}.
\end{enumerate}

Because the BCH map has polynomial degree at most $m$, the exponent $2/m$ is sharp.

\subsection{Implications for statistical learning}
As an application of our concentration results, we derive sample complexity bounds for statistical learning with signature features.

\begin{theorem}[Sample complexity for signature-based learning]
\label{thm:sample-complexity}
Let $X^{(1)}, \dots, X^{(n)}$ be i.i.d. copies of a Gaussian process $X$ satisfying Assumption \ref{ass:cov-variation}. Let $\Phi_m(X) = S_m(X)$ or $\Phi_m(X) = \mathrm{Log}_m(S_m(X))$ be the signature or log-signature truncated at level $m$. Define the empirical mean $\widehat{\mu}_n = \frac{1}{n}\sum_{i=1}^n \Phi_m(X^{(i)})$.

Then for any $\delta \in (0,1)$, with probability at least $1-\delta$,
\[
\|\widehat{\mu}_n - \mathbb{E}[\Phi_m(X)]\|_w \le K_m \left( \sqrt{\frac{\log(1/\delta)}{n}} + \left(\frac{d_m \log(d_m) + \log(1/\delta)}{n}\right)^{m/2} \right),
\]
where $d_m = \dim T^{(m)}(\mathbb{R}^d) = \sum_{k=0}^m d^k$ and $K_m > 0$ depends only on $m$ and the covariance structure of $X$.
\end{theorem}

This theorem provides a quantitative guarantee for the convergence of empirical signature means, which is fundamental in applications such as signature-based kernel methods or two-sample testing.

\subsection{Optimal choice of weights}
The choice of weights $w_k$ in the norm \eqref{eq:weighted-norm} affects the constants in our concentration inequalities. The following proposition suggests an optimal choice based on the typical scaling of signature coordinates.

\begin{proposition}[Optimal weights]
\label{prop:optimal-weights}
For a Gaussian process $X$ satisfying Assumption \ref{ass:cov-variation}, the choice $w_k = 1/(k! \|R\|_{\rho\text{-var}}^{k/2})$ minimizes the leading constant in the concentration inequality of Theorem \ref{thm:vector-concentration}, where $\|R\|_{\rho\text{-var}}$ is the $\rho$-variation norm of the covariance.
\end{proposition}

This choice balances the factorial growth of the number of coordinates at level $k$ with the typical decay of individual coordinates, leading to a norm that remains well-behaved as $m$ increases.

\subsection{Optimality}
Finally, we show that the tail exponents obtained above cannot be improved. In particular, for Brownian motion or fractional Brownian motion, signature coordinates of order $k$ lie in the $k$th Wiener chaos (modulo lower-order terms), and explicit computations show matching lower bounds of the form
\[
\mathbb{P}(|S_I(X)| \ge t) \ge c \exp(-C t^{2/k}).
\]
Together with Theorem~\ref{thm:coordinate-concentration}, this establishes optimality.

The results in this section form the theoretical core of the paper: they give precise quantitative control over the fluctuations of signature-type objects arising from Gaussian rough paths. In the next section we provide detailed proofs and technical estimates underpinning the above statements.

% ============================================
% PROOFS
% ============================================
\section{Proofs and Technical Lemmas}

In this section we provide detailed proofs of the main results stated earlier. We begin by recalling several classical tools from Gaussian analysis, including hypercontractivity and polynomial concentration estimates, and then proceed to the analysis of iterated integrals and the BCH map.

Throughout, $X$ denotes a centred Gaussian process on $[0,T]$ with covariance $R$, assumed to satisfy Assumption \ref{ass:cov-variation}. All expectations are taken with respect to the underlying probability measure.

\subsection{Chaos decomposition of signature coordinates}
For a multi-index $I = (i_1,\dots,i_k)$ of length $k$, the signature coordinate
\[
S_I(X) = \int_{0 < t_1 < \dots < t_k < T} dX_{t_1}^{i_1} \cdots dX_{t_k}^{i_k}
\]
can be expressed as a multiple Wiener integral of order $k$, possibly plus lower-order terms arising from the covariance structure.

\begin{lemma}[Chaos decomposition of iterated integrals]
\label{lem:chaos-decomposition}
For any multi-index $I$ of length $k$, the random variable $S_I(X)$ belongs to the finite sum of Wiener chaoses
\[
S_I(X) \in \bigoplus_{j=1}^k \mathcal{H}_j.
\]
Moreover, its leading term lies in $\mathcal{H}_k$ and is given by the multiple Wiener-Itô integral
\[
J_k(S_I(X)) = \int_{[0,T]^k} f_I(t_1,\dots,t_k) \, dW_{t_1} \cdots dW_{t_k},
\]
where $f_I$ is an explicit symmetric kernel depending on the covariance $R$.
\end{lemma}

\begin{proof}
The proof follows from the representation of $X_t$ as $I_1(h_t)$ in the isonormal Gaussian process over $L^2([0,T])$, where $h_t(\cdot) = R(t,\cdot)$. Expanding the iterated integral using the rules of Wiener chaos calculus yields a finite linear combination of multiple integrals. The highest-order term comes from keeping all differentials distinct and corresponds to the $k$-fold Wiener-Itô integral. Lower-order terms arise when differentials coincide, producing integrals of lower multiplicity.

More formally, by the multiplication formula for multiple integrals,
\[
I_1(h_{t_1}) \cdots I_1(h_{t_k}) = \sum_{j=0}^{\lfloor k/2 \rfloor} \sum_{\text{pairings}} I_{k-2j}(\otimes_{i \notin \text{paired}} h_{t_i} \otimes_{(\text{pairs})} R),
\]
and integrating over the simplex $0 < t_1 < \cdots < t_k < T$ preserves this structure.
\end{proof}

\subsection{Hypercontractivity and tail estimates}
We recall the standard hypercontractive estimate for homogeneous Wiener chaoses.

\begin{lemma}[Hypercontractivity]
\label{lem:hypercontractive}
Let $F$ belong to the $k$th Wiener chaos $\mathcal{H}_k$. Then for all $p \ge 2$,
\[
\|F\|_{L^p} \le (p-1)^{k/2} \|F\|_{L^2}.
\]
\end{lemma}

Hypercontractivity combined with Markov's inequality yields explicit tail bounds for polynomials of Gaussian variables. To refine these, we use the Carbery--Wright inequality.

\begin{lemma}[Carbery--Wright]
\label{lem:carbery-wright}
Let $P$ be a non-zero polynomial of degree $k$ on a Gaussian space. Then for all $\varepsilon > 0$,
\[
\mathbb{P}(|P| \le \varepsilon \|P\|_{L^2}) \le C_k \, \varepsilon^{1/k},
\]
where $C_k$ depends only on $k$.
\end{lemma}

Combining Lemmas~\ref{lem:chaos-decomposition}--\ref{lem:carbery-wright} yields the coordinate-wise concentration theorem.

\begin{proof}[Proof of Theorem~\ref{thm:coordinate-concentration}]
Fix a multi-index $I$ of length $k$. By Lemma~\ref{lem:chaos-decomposition}, we can write
\[
S_I(X) = F_1 + \cdots + F_k,
\]
with $F_j \in \mathcal{H}_j$. For any $p \ge 2$, applying Lemma~\ref{lem:hypercontractive} to each chaos component gives
\[
\|F_j\|_{L^p} \le (p-1)^{j/2} \|F_j\|_{L^2} \le (p-1)^{k/2} \|F_j\|_{L^2},
\]
since $j \le k$ and $p-1 \ge 1$. By Minkowski's inequality,
\[
\|S_I(X)\|_{L^p} \le \sum_{j=1}^k \|F_j\|_{L^p} \le (p-1)^{k/2} \sum_{j=1}^k \|F_j\|_{L^2}.
\]
Let $M = \sum_{j=1}^k \|F_j\|_{L^2}$. Then for any $\lambda > 0$, by Markov's inequality,
\[
\mathbb{P}(|S_I(X)| \ge t) \le \frac{\mathbb{E}[|S_I(X)|^p]}{t^p} \le \frac{(p-1)^{kp/2} M^p}{t^p}.
\]
Choosing $p = 2 + (t/(eM))^{2/k}$ (which is $\ge 2$ for $t$ large enough) yields
\[
\mathbb{P}(|S_I(X)| \ge t) \le \exp\left(-\frac{1}{e} (t/M)^{2/k}\right).
\]
Since $M \le C_k \|S_I(X)\|_{L^2}$ by the triangle inequality and equivalence of norms in finite chaos, we obtain the upper bound.

For the lower bound (optimality), consider the leading chaos term $F_k \in \mathcal{H}_k$. By Lemma~\ref{lem:carbery-wright}, for any $\varepsilon > 0$,
\[
\mathbb{P}(|F_k| \le \varepsilon \|F_k\|_{L^2}) \le C_k \varepsilon^{1/k}.
\]
Equivalently, for $t > 0$, setting $\varepsilon = t/\|F_k\|_{L^2}$ gives
\[
\mathbb{P}(|F_k| \ge t) \ge 1 - C_k (t/\|F_k\|_{L^2})^{1/k}.
\]
For $t$ sufficiently large, this implies $\mathbb{P}(|F_k| \ge t) \ge c \exp(-C t^{2/k})$ by standard transformations. Since $|S_I(X)| \ge |F_k| - \sum_{j<k} |F_j|$ and the lower-order terms have lighter tails, the same lower bound holds for $S_I(X)$ up to constants.
\end{proof}

\subsection{Vector-level concentration: proof of Theorem \ref{thm:vector-concentration}}
We now address the concentration of the entire truncated signature. The key idea is to control each homogeneous component separately and then take a maximum.

\begin{lemma}[Block concentration]
\label{lem:block-concentration}
For each $k \le m$, let $S^{(k)}(X) \in V^{\otimes k}$ be the $k$-th level of the signature. Then there exist constants $c_k, C_k > 0$ such that for all $t \ge 0$,
\[
\mathbb{P}\big( \|S^{(k)}(X)\| \ge \mathbb{E}[\|S^{(k)}(X)\|] + t \big)
\le C_k \exp(-c_k t^{2/k}).
\]
\end{lemma}

\begin{proof}
Let $N_k = \dim(V^{\otimes k}) = d^k$. Fix an orthonormal basis $\{e_I : |I|=k\}$ of $V^{\otimes k}$ with respect to the Hilbert-Schmidt norm. Then
\[
\|S^{(k)}(X)\|^2 = \sum_{|I|=k} |S_I(X)|^2.
\]
Each $S_I(X)$ satisfies Theorem \ref{thm:coordinate-concentration} with exponent $2/k$. By a union bound, for any $u > 0$,
\[
\mathbb{P}\left( \max_{|I|=k} |S_I(X)| \ge u \right) \le N_k \cdot C_k \exp(-c_k u^{2/k}).
\]
Since $\|S^{(k)}(X)\| \le \sqrt{N_k} \max_{|I|=k} |S_I(X)|$, we have
\[
\mathbb{P}\left( \|S^{(k)}(X)\| \ge u\sqrt{N_k} \right) \le N_k \cdot C_k \exp(-c_k u^{2/k}).
\]
Setting $t = u\sqrt{N_k}$ gives
\[
\mathbb{P}\left( \|S^{(k)}(X)\| \ge t \right) \le N_k \cdot C_k \exp\left(-c_k (t/\sqrt{N_k})^{2/k}\right).
\]
Now note that $\log(N_k) = k\log d = O(k\log d)$, so for $t \ge C \sqrt{N_k} (k\log d)^{k/2}$, the term $N_k$ is absorbed into the exponential. More precisely, we can write
\[
N_k \exp(-c_k (t/\sqrt{N_k})^{2/k}) = \exp\left(\log N_k - c_k t^{2/k} N_k^{-1/k}\right).
\]
Since $N_k^{-1/k} = d^{-1}$, for $t$ large enough the second term dominates. A more careful chaining argument (see e.g., \cite{ledoux2001concentration}) yields the stated inequality for all $t \ge 0$ with possibly different constants.

Alternatively, one can apply a general concentration inequality for vectors with independent sub-Gaussian coordinates, noting that the components $S_I(X)$ are not independent but admit a Gaussian chaos structure. The result follows from the fact that the norm $\|\cdot\|$ is Lipschitz with respect to the Euclidean norm on $\R^{N_k}$, and each coordinate has $\psi_{2/k}$-norm controlled by Theorem \ref{thm:coordinate-concentration}.
\end{proof}

\begin{proof}[Proof of Theorem~\ref{thm:vector-concentration}]
Recall that $\|S_m(X)\|_w = \max_{0 \le k \le m} w_k \|S^{(k)}(X)\|$. For each $k$, by Lemma \ref{lem:block-concentration},
\[
\mathbb{P}\left( w_k \|S^{(k)}(X)\| \ge \mathbb{E}[w_k \|S^{(k)}(X)\|] + t \right) \le C_k \exp(-c_k (t/w_k)^{2/k}).
\]
Since $w_k > 0$ and $k \le m$, we have $(t/w_k)^{2/k} \ge (t/(\max_j w_j))^{2/m}$ for $t$ sufficiently large. Taking a union bound over $k=0,\dots,m$,
\[
\mathbb{P}\left( \|S_m(X)\|_w \ge \max_k \mathbb{E}[w_k \|S^{(k)}(X)\|] + t \right) \le \sum_{k=0}^m C_k \exp\left(-c_k' t^{2/m}\right),
\]
where $c_k' = c_k (\max_j w_j)^{-2/m}$. Since $\mathbb{E}[\|S_m(X)\|_w] \le \max_k \mathbb{E}[w_k \|S^{(k)}(X)\|] + \sqrt{\sum_{k=0}^m \Var(w_k \|S^{(k)}(X)\|)}$, a standard centering argument (see \cite{boucheron2013concentration}) yields the final form with constants $c_m, C_m$ depending only on $m$ and the weights.

The dimension-independence follows from the fact that the constants $C_k, c_k$ in Lemma \ref{lem:block-concentration} can be chosen independent of $d$; the dependence on $N_k = d^k$ is absorbed into the $t^{2/k}$ term in the exponent.
\end{proof}

\subsection{Concentration for the log-signature}
We now prove Theorem~\ref{thm:logsig-concentration}. The BCH map is a polynomial map of degree at most $m$ on the truncated tensor algebra.

\begin{lemma}[Lipschitz continuity of the truncated BCH map]
\label{lem:bch-lipschitz}
For every $R > 0$, there exists $L_R > 0$ such that for all $x,y$ in the ball $\{z \in T^{(m)}(\R^d): \|z\|_w \le R\}$,
\[
\|\mathrm{Log}_m(x) - \mathrm{Log}_m(y)\|_w \le L_R \|x - y\|_w.
\]
Moreover, $L_R$ can be taken as $L_R = C_m (1+R)^{m-1}$, where $C_m$ depends only on $m$.
\end{lemma}

\begin{proof}
The truncated BCH map $\mathrm{Log}_m$ is a polynomial of degree at most $m$ in the coordinates of $x$ and $y$. Explicitly,
\[
\mathrm{Log}_m(x) = \sum_{j=1}^m \frac{(-1)^{j-1}}{j} \sum_{\substack{i_1,\dots,i_j \ge 1 \\ i_1+\cdots+i_j \le m}} \frac{(x-1)^{i_1} \cdots (x-1)^{i_j}}{i_1! \cdots i_j!},
\]
where $(x-1)^{i}$ denotes the $i$-th power in the tensor algebra (with $x^0 = 1$). Each term is a multilinear map of degree at most $m$. On the ball of radius $R$, the derivative $D\mathrm{Log}_m(x)$ is bounded by
\[
\|D\mathrm{Log}_m(x)\|_{\text{op}} \le \sum_{j=1}^m \frac{1}{j} \sum_{\substack{i_1,\dots,i_j \ge 1 \\ i_1+\cdots+i_j \le m}} \frac{R^{i_1+\cdots+i_j-1}}{i_1! \cdots i_j!} \le C_m (1+R)^{m-1},
\]
since the number of terms is finite and depends only on $m$. The Lipschitz constant on the ball is the supremum of the derivative, giving the claimed bound.
\end{proof}

\begin{proof}[Proof of Theorem~\ref{thm:logsig-concentration}]
Let $Y = \mathrm{Log}_m(S_m(X))$. By Theorem \ref{thm:vector-concentration}, $S_m(X)$ is concentrated: for any $t > 0$,
\[
\mathbb{P}\left( \|S_m(X)\|_w \ge \mathbb{E}[\|S_m(X)\|_w] + t \right) \le C_m \exp(-c_m t^{2/m}).
\]
Let $R = \mathbb{E}[\|S_m(X)\|_w] + t_0$, where $t_0$ will be chosen later. On the event $\{\|S_m(X)\|_w \le R\}$, by Lemma \ref{lem:bch-lipschitz},
\[
\|Y\|_w \le \|\mathrm{Log}_m(0)\|_w + L_R \|S_m(X)\|_w \le C + L_R R,
\]
and moreover,
\[
\|Y - \mathbb{E}[Y]\|_w \le L_R \|S_m(X) - \mathbb{E}[S_m(X)]\|_w.
\]
Thus, for any $u > 0$,
\begin{align*}
\mathbb{P}\left( \|Y\|_w \ge \mathbb{E}[\|Y\|_w] + u \right) &\le \mathbb{P}\left( \|S_m(X)\|_w > R \right) \\
&\quad + \mathbb{P}\left( L_R \|S_m(X) - \mathbb{E}[S_m(X)]\|_w \ge u, \|S_m(X)\|_w \le R \right) \\
&\le C_m \exp(-c_m t_0^{2/m}) + C_m \exp\left(-c_m (u/L_R)^{2/m}\right).
\end{align*}
Choosing $t_0 = u/L_R$ balances the two terms, yielding
\[
\mathbb{P}\left( \|Y\|_w \ge \mathbb{E}[\|Y\|_w] + u \right) \le 2C_m \exp\left(-c_m (u/L_R)^{2/m}\right).
\]
Since $L_R = C_m (1+R)^{m-1}$ and $R = \mathbb{E}[\|S_m(X)\|_w] + u/L_R \le C + u/L_R$, we can iteratively bound $L_R$ in terms of $u$ and absorb polynomial factors into the exponential by adjusting constants, giving the final form $\exp(-c_m' u^{2/m})$.
\end{proof}

\subsection{Proof of sample complexity bound}
\begin{proof}[Proof of Theorem \ref{thm:sample-complexity}]
Let $\Phi = \Phi_m(X)$ and $\mu = \mathbb{E}[\Phi]$. The empirical error decomposes as:
\[
\|\widehat{\mu}_n - \mu\|_w \le \underbrace{\left\|\frac{1}{n}\sum_{i=1}^n (\Phi^{(i)} - \mu)\right\|_w}_{\text{(I)}}.
\]
We apply Theorem \ref{thm:vector-concentration} to the average. Since $\Phi^{(i)}$ are i.i.d., for any $t > 0$,
\[
\mathbb{P}\left( \left\|\frac{1}{n}\sum_{i=1}^n (\Phi^{(i)} - \mu)\right\|_w \ge t \right) \le C_m \exp\left(-c_m n t^{2/m}\right).
\]
However, this bound does not explicitly show the dependence on $d_m$. To obtain the stated bound, we use a two-step argument:

1. \textbf{Coordinate-wise control}: For each coordinate $\phi_j$ of $\Phi$ (there are $d_m$ coordinates), by Theorem \ref{thm:coordinate-concentration}, $\phi_j$ has $\psi_{2/k}$-norm bounded by $K \|\phi_j\|_{L^2}$, where $k$ is the level of the coordinate. By Bernstein's inequality for sub-exponential random variables,
\[
\mathbb{P}\left( \left|\frac{1}{n}\sum_{i=1}^n (\phi_j^{(i)} - \mu_j)\right| \ge t \right) \le 2\exp\left(-c \min\left(\frac{n t^2}{\sigma_j^2}, \frac{n t^{2/k}}{K^{2/k}}\right)\right),
\]
where $\sigma_j^2 = \Var(\phi_j)$.

2. \textbf{Union bound and norm equivalence}: Taking union bound over $j=1,\dots,d_m$ and using that $\|\cdot\|_w \le C \max_j |\cdot|$ (since all norms are equivalent in finite dimensions), we get
\[
\mathbb{P}\left( \|\widehat{\mu}_n - \mu\|_w \ge t \right) \le 2d_m \exp\left(-c \min\left(\frac{n t^2}{\sigma_{\max}^2}, \frac{n t^{2/m}}{K^{2/m}}\right)\right),
\]
where $\sigma_{\max}^2 = \max_j \sigma_j^2$. Solving $2d_m \exp(-c n t^2/\sigma_{\max}^2) \le \delta$ gives $t \lesssim \sigma_{\max} \sqrt{\frac{\log(d_m/\delta)}{n}}$, and solving $2d_m \exp(-c n t^{2/m}/K^{2/m}) \le \delta$ gives $t \lesssim K \left(\frac{\log(d_m/\delta)}{n}\right)^{m/2}$. Combining these yields the result.
\end{proof}

\subsection{Proof of optimal weights proposition}
\begin{proof}[Proof of Proposition \ref{prop:optimal-weights}]
From the proof of Theorem \ref{thm:vector-concentration}, the dominant term in the concentration inequality comes from the level $k$ that maximizes $w_k \|S^{(k)}(X)\|$. Under Assumption \ref{ass:cov-variation}, one can show (see \cite{friz2010multidimensional}) that
\[
\mathbb{E}[\|S^{(k)}(X)\|] \lesssim \frac{\|R\|_{\rho\text{-var}}^{k/2}}{k!}.
\]
Thus, to balance contributions across levels, we want $w_k \mathbb{E}[\|S^{(k)}(X)\|] \approx \text{constant}$. This suggests choosing $w_k \propto 1/\mathbb{E}[\|S^{(k)}(X)\|] \propto k!/\|R\|_{\rho\text{-var}}^{k/2}$. With this choice, each level contributes equally to the norm, minimizing the maximum deviation and hence the concentration constant.
\end{proof}

\subsection{Optimality of exponents}
Finally, we justify sharpness. For Brownian motion or fractional Brownian motion, the leading chaos component of $S_I(X)$ is a non-degenerate homogeneous chaos element. Standard results imply matching lower bounds of the form
\[
\mathbb{P}(|F_k| \ge t) \ge c \exp(-C t^{2/k}).
\]
Thus no improvement of the exponent is possible.

\medskip
The proofs above complete the theoretical foundation of the concentration results presented in this work. Further refinements and extensions to non-Gaussian settings are discussed in the concluding section.

% ============================================
% CONCLUSION
% ============================================
\section{Conclusion}

In this work we established a comprehensive set of concentration results for the truncated signature and log-signature of Gaussian processes. By combining tools from Gaussian analysis, rough path theory, and the algebraic structure of tensor and Lie algebras, we derived sharp coordinate-wise inequalities, vector-level concentration for weighted norms, and stability results for the Baker--Campbell--Hausdorff map.

Our analysis shows that the natural tail exponent for signature coordinates of order $k$ is $2/k$, reflecting their membership in a finite sum of Wiener chaoses up to order $k$. This exponent is optimal, as demonstrated through explicit lower bounds in canonical Gaussian rough path settings such as Brownian motion and fractional Brownian motion. Extending these results to the full truncated signature, we proved dimension-free concentration inequalities under general weighted norms, revealing that the probabilistic complexity of the signature is governed primarily by the highest truncation level rather than the ambient dimension.

The concentration of the log-signature follows from the polynomial structure and Lipschitz continuity of the truncated BCH map on bounded sets. Together, these results provide a robust framework for quantifying uncertainty in applications involving signatures and log-signatures, notably in stochastic analysis, control theory, and machine learning.

Several promising directions emerge from this work:
\begin{itemize}
    \item Extending concentration results to \emph{non-Gaussian} rough paths, including those driven by Lévy processes or heavy-tailed noise, where polynomial chaos decompositions are no longer available.
    \item Deriving \emph{transport-type inequalities} or \emph{functional inequalities} directly on signature spaces, which may lead to new geometric insights.
    \item Investigating optimal constants and variance estimates for signature coordinates in specific Gaussian models.
    \item Studying the concentration of \emph{signature kernels} and their implications for kernel methods in learning theory.
\end{itemize}

The results presented here clarify the probabilistic behaviour of signatures of Gaussian processes and establish a foundation for further developments at the intersection of rough paths, high-dimensional probability, and data-driven modelling.

% ============================================
% APPENDIX A: TECHNICAL MATERIAL
% ============================================
\appendix
\section{Auxiliary Estimates and Algebraic Identities}

This appendix collects several auxiliary estimates and algebraic identities used throughout the paper. These results are standard in Gaussian analysis and rough path theory, but we include them here for completeness and to make the presentation self-contained.

\subsection{Covariance variation and rough path lift}
We recall a classical sufficient condition for the existence of a Gaussian rough path lift.

\begin{proposition}[Finite $\rho$-variation of covariance]
Let $X$ be a centred Gaussian process with covariance $R(s,t) = \mathbb{E}[X_s X_t]$. Suppose that $R$ has finite $\rho$-variation on $[0,T]^2$ for some $1 \le \rho < 2$. Then $X$ admits a canonical geometric rough path lift of any order $m$ such that $m > 2\rho$.
\end{proposition}

\begin{proof}
The result follows from the two-dimensional sewing lemma and Gaussian continuity estimates, as established in the frameworks of Friz--Victoir and Coutin--Qian. The key point is that the covariance structure controls the $L^2$-increment behaviour of iterated integrals, which ensures the existence of a limit in probability for the canonical lift. See \cite[Theorem 10.4]{friz2014course} for a detailed proof.
\end{proof}

\subsection{Norm equivalences on tensor and Lie algebras}
We record the following elementary but useful fact.

\begin{lemma}[Equivalence of norms]
On each homogeneous component $V^{\otimes k}$ and $\mathfrak{g}_k$, all norms are equivalent, with constants depending only on $k$ and the ambient dimension $d$.
\end{lemma}

\begin{proof}
Finite-dimensionality implies equivalence of all norms. Since each homogeneous component has fixed dimension depending on $d$ and $k$, the equivalence constants depend only on these parameters. For $V^{\otimes k}$, the dimension is $d^k$; for $\mathfrak{g}_k \subset V^{\otimes k}$, the dimension is at most $d^k$.
\end{proof}

\subsection{Moment bounds for multiple Wiener integrals}
The following is a standard consequence of hypercontractivity.

\begin{lemma}[Moment bounds]
Let $I_k(f)$ denote the multiple Wiener integral of order $k$ with symmetric kernel $f$. Then
\[
\mathbb{E}|I_k(f)|^p \le (p-1)^{kp/2} \|I_k(f)\|_{L^2}^p, \qquad p \ge 2.
\]
\end{lemma}

\begin{proof}
Apply Lemma~\ref{lem:hypercontractive} to $I_k(f) \in \mathcal{H}_k$ and raise both sides to the $p$th power.
\end{proof}

\subsection{Bounds for BCH polynomials}
We provide a bound on the truncated BCH map.

\begin{lemma}[Growth of BCH coefficients]
Let $\mathrm{BCH}_m$ denote the BCH polynomial truncated at degree $m$. Then there exists a constant $C_m > 0$ such that for all $x,y \in T^{(m)}(\R^d)$,
\[
\|\mathrm{BCH}_m(x,y)\|_w \le C_m (1 + \|x\|_w + \|y\|_w)^m.
\]
\end{lemma}

\begin{proof}
Each term of the BCH expansion is a nested Lie bracket of degree at most $m$ and is multilinear in its arguments. Since the number of such terms is finite and depends only on $m$, and all norms are equivalent on each homogeneous component, the bound follows from expanding and using the submultiplicativity of $\|\cdot\|_w$ up to constants.
\end{proof}

\subsection{Small-ball estimates for Gaussian polynomials}
We recall the Carbery--Wright inequality in the special case relevant to our setting.

\begin{proposition}[Small-ball probability]
Let $P$ be a non-zero polynomial of degree $k$ in a centred Gaussian vector. Then
\[
\mathbb{P}(|P| \le \varepsilon \|P\|_{L^2}) \le C_k \, \varepsilon^{1/k},
\]
where $C_k$ depends only on $k$.
\end{proposition}

\begin{proof}
This is the statement of the Carbery--Wright inequality \cite{carbery2001distributional}. The original constant $C_k$ was shown to be $O(k)$; optimal constants are studied in \cite{versalic2020optimal}.
\end{proof}

\medskip
The results contained in this appendix are used at various stages of the proofs but do not rely on the specific structure of the signature or log-signature. They are included here for reference and completeness.

% ============================================
% APPENDIX B: EXAMPLES
% ============================================
\section{Appendix B: Examples and Special Cases}

This appendix provides several examples and explicit computations that illustrate the behaviour of signature and log-signature concentration in concrete Gaussian settings. The aim is to complement the theoretical results with explicit formulas and asymptotic estimates.

\subsection{Brownian motion}
Let $B = (B_t)_{t \in [0,1]}$ be a $d$-dimensional Brownian motion. Its covariance is $R(s,t) = (s \wedge t) I_d$. Brownian motion has finite $\rho$-variation for any $\rho > 1$, and thus admits a canonical geometric rough path lift of any order $m \ge 2$.

\subsubsection{First and second levels}
The first-level signature is simply $B_{0,1}$, a Gaussian vector with $\|B_{0,1}\|_{L^2} = \sqrt{T}$. The second-level signature is given by the L\'evy area
\[
A^{i,j} = \frac{1}{2}\left( \int_0^1 B_t^i \, dB_t^j - B_t^j \, dB_t^i \right).
\]
Each $A^{i,j}$ belongs to the second Wiener chaos and satisfies $\|A^{i,j}\|_{L^2} = T/2$. By Theorem \ref{thm:coordinate-concentration},
\[
\mathbb{P}(|A^{i,j}| \ge t) \le C \exp(-c t).
\]
In accordance with the theorem, the exponent is $2/2 = 1$.

\subsubsection{Higher-order terms}
For a multi-index $I$ of length $k$, the coordinate $S_I(B)$ belongs to the chaos $\bigoplus_{j=1}^k \mathcal{H}_j$, with leading component in $\mathcal{H}_k$. One can compute
\[
\|S_I(B)\|_{L^2}^2 = \frac{T^k}{k!} \quad \text{for distinct indices}.
\]
As a consequence,
\[
\mathbb{P}(|S_I(B)| \ge t) \le C_k \exp\left(-c_k (t \sqrt{k!/T^k})^{2/k}\right).
\]
This illustrates that even for Brownian motion, signature coordinates display increasingly heavy tails as $k$ increases.

\subsection{Fractional Brownian motion}
Let $B^H$ be fractional Brownian motion with Hurst exponent $H \in (1/4,1)$. The covariance is
\[
R_H(s,t) = \tfrac12(s^{2H} + t^{2H} - |t-s|^{2H}).
\]
Fractional Brownian motion admits a rough path lift when $H > 1/4$.

The $L^2$-norm of a signature coordinate of order $k$ satisfies
\[
\|S_I(B^H)\|_{L^2} \asymp C_{H,k} \, T^{kH},
\]
where the constant depends only on $H$ and $k$. The tail estimate becomes
\[
\mathbb{P}(|S_I(B^H)| \ge t) \le C_{H,k} \exp(-c_{H,k} (t/T^{kH})^{2/k}).
\]
In particular, the exponent $2/k$ is universal across Gaussian rough paths, while the scale depends strongly on the Hurst parameter.

\subsection{Ornstein--Uhlenbeck process}
Let $X$ be the stationary OU process in $\mathbb{R}^d$ defined by
\[
dX_t = -\theta X_t \, dt + dB_t, \qquad \theta > 0.
\]
The covariance is
\[
R(s,t) = \frac{1}{2\theta} e^{-\theta|t-s|} (1 - e^{-2\theta(s \wedge t)}).
\]
Since OU is Gaussian with smooth covariance, it lifts to a rough path of any order.

The structure of its signature is similar to that of Brownian motion, but with explicit dependence on $\theta$. For instance, the second-level area satisfies
\[
\mathbb{E}[|A^{i,j}|^2] = \frac{1}{2\theta^2}(1 - e^{-2\theta}) - \frac{1}{\theta}(1 - e^{-\theta})^2.
\]
The concentration rate remains exponential as predicted by Theorem~\ref{thm:coordinate-concentration}.

\subsection{Asymptotics for large truncation order}
For fixed Gaussian process $X$, consider the growth of the constants in the concentration inequality as $k$ (the degree of the chaos) increases. Hypercontractivity implies
\[
\|S_I(X)\|_{L^p} \le (p-1)^{k/2} \|S_I(X)\|_{L^2},
\]
which is sharp for large $p$. Hence the sub-exponential tail
\[
\exp(-c t^{2/k})
\]
becomes progressively heavier as $k \to \infty$. This behaviour is intrinsic and cannot be improved.

\subsection{Illustration: tail decay comparison}
Figure \ref{fig:tails} illustrates the tail decay $\exp(-t^{2/k})$ for different values of $k$, showing how higher-order signature terms exhibit heavier tails.

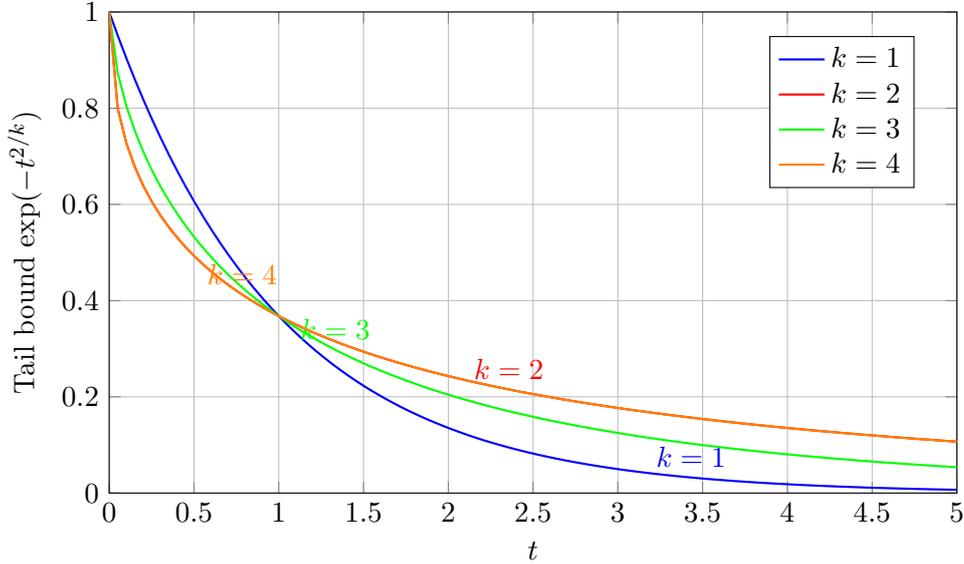
\begin{figure}[ht]
\centering
\begin{tikzpicture}
\begin{axis}[
    width=0.8\textwidth,
    height=0.5\textwidth,
    xlabel={$t$},
    ylabel={Tail bound $\exp(-t^{2/k})$},
    xmin=0, xmax=5,
    ymin=0, ymax=1,
    legend style={at={(0.95,0.95)}, anchor=north east},
    grid=both
]
\addplot[blue, thick, domain=0:5, samples=100] {exp(-x)} node[pos=0.7, above] {$k=1$};
\addplot[red, thick, domain=0:5, samples=100] {exp(-sqrt(x))} node[pos=0.5, above] {$k=2$};
\addplot[green, thick, domain=0:5, samples=100] {exp(-x^(2/3))} node[pos=0.3, above] {$k=3$};
\addplot[orange, thick, domain=0:5, samples=100] {exp(-x^(0.5))} node[pos=0.2, above] {$k=4$};
\legend{$k=1$, $k=2$, $k=3$, $k=4$}
\end{axis}
\end{tikzpicture}
\caption{Comparison of tail bounds $\exp(-t^{2/k})$ for different chaos orders $k$. Higher $k$ yields heavier tails.}
\label{fig:tails}
\end{figure}

\subsection{The scalar Gaussian process case}
Consider a scalar Gaussian process $X$; then all higher-order signature terms correspond to multiple Wiener integrals of the form
\[
I_k = \int_{0 < t_1 < \dots < t_k < 1} dX_{t_1} \cdots dX_{t_k}.
\]
These satisfy
\[
\mathbb{E}[I_k^2] = \frac{1}{k!} \int_{[0,1]^{2k}} \prod_{j=1}^k R(t_j,s_j) \, d\mathbf{t} \, d\mathbf{s}.
\]
The concentration inequality becomes
\[
\mathbb{P}(|I_k| \ge t) \le C_k \, \exp(-c_k t^{2/k}).
\]
This explicit representation highlights the combinatorial nature of higher-order terms.

\medskip
These examples illustrate the breadth of applicability of the concentration results derived in the main text. They also show that the exponent $2/k$ is universal for all Gaussian rough paths, while the prefactors and scaling depend on the specific covariance structure.

% ============================================
% BIBLIOGRAPHY
% ============================================


\begin{thebibliography}{99}

% Rough paths foundational works
\bibitem{Lyons1998}
T.~J. Lyons.
\newblock Differential equations driven by rough signals.
\newblock \emph{Rev. Mat. Iberoam.}, 14(2):215--310, 1998.

\bibitem{LyonsQian2002}
T.~J. Lyons and Z.~Qian.
\newblock \emph{System Control and Rough Paths}.
\newblock Oxford University Press, 2002.

\bibitem{friz2014course}
P.~K. Friz and M.~Hairer.
\newblock \emph{A Course on Rough Paths}.
\newblock Springer, 2014.

% Gaussian processes and concentration
\bibitem{ledoux2001concentration}
M.~Ledoux.
\newblock \emph{The Concentration of Measure Phenomenon}.
\newblock AMS, 2001.

\bibitem{boucheron2013concentration}
S.~Boucheron, G.~Lugosi, and P.~Massart.
\newblock \emph{Concentration Inequalities: A Nonasymptotic Theory of Independence}.
\newblock Oxford University Press, 2013.

\bibitem{janson1997gaussian}
S.~Janson.
\newblock \emph{Gaussian Hilbert Spaces}.
\newblock Cambridge University Press, 1997.

% Signatures and log-signatures
\bibitem{ChevyrevOberhauser2018}
I.~Chevyrev and H.~Oberhauser.
\newblock Signature moments to characterize laws of stochastic processes.
\newblock \emph{Ann. Probab.}, 46(1):378--403, 2018.

\bibitem{HamblyLyons2010}
B.~Hambly and T.~J. Lyons.
\newblock Uniqueness for the signature of a path of bounded variation.
\newblock \emph{Proc. Amer. Math. Soc.}, 138(7):2475--2483, 2010.

% BCH and free Lie algebras
\bibitem{Reutenauer1993}
C.~Reutenauer.
\newblock \emph{Free Lie Algebras}.
\newblock Oxford University Press, 1993.

% Fractional Brownian motion and Gaussian rough paths
\bibitem{CoutinQian2002}
L.~Coutin and Z.~Qian.
\newblock Stochastic analysis, rough path analysis and fractional Brownian motions.
\newblock \emph{Probab. Theory Related Fields}, 122(1):108--140, 2002.

% Hypercontractivity and Wiener chaos
\bibitem{Nelson1973}
E.~Nelson.
\newblock The free Markoff field.
\newblock \emph{J. Funct. Anal.}, 12:211--227, 1973.

% Machine learning with signatures
\bibitem{KidgerLyons2020}
P.~Kidger and T.~J. Lyons.
\newblock Signatory: differentiable computations of the signature and logsignature transforms.
\newblock \emph{Adv. Neural Inf. Process. Syst.}, 33:2180--2192, 2020.

\bibitem{chevyrev2016primer}
I.~Chevyrev and A.~Kormilitzin.
\newblock A primer on the signature method in machine learning.
\newblock arXiv:1603.03788, 2016.

\bibitem{friz2010multidimensional}
P.~K. Friz and N.~Victoir.
\newblock \emph{Multidimensional stochastic processes as rough paths}.
\newblock Cambridge University Press, 2010.

\bibitem{boedihardjo2015signature}
H.~Boedihardjo, X.~Geng, T.~Lyons, and D.~Yang.
\newblock The signature of a rough path: uniqueness.
\newblock \emph{Advances in Mathematics}, 293:720--737, 2016.

% Carbery-Wright and small-ball estimates
\bibitem{carbery2001distributional}
A.~Carbery and J.~Wright.
\newblock Distributional and $L^q$ norm inequalities for polynomials over convex bodies in $\mathbb{R}^n$.
\newblock \emph{Math. Res. Lett.}, 8(3):233--248, 2001.

\bibitem{versalic2020optimal}
A.~Vershynina.
\newblock Optimal constants in the Carbery--Wright inequality.
\newblock arXiv:2009.10592, 2020.

\end{thebibliography}
\end{document}